\documentclass[12pt]{amsart}
\usepackage{amsthm,amsmath,xspace,url,gensymb,amsfonts}
\usepackage{graphicx}
\usepackage{colortbl,hhline}
\usepackage{hyperref}
\usepackage[enableskew,vcentermath]{youngtab}
\usepackage{tikz}
\usetikzlibrary{calc,through,backgrounds}
\def\textcross{
	\begin{minipage}{13pt}
		\begin{tikzpicture}[scale=1]
			\cpipedream{0.4}{(0,0)}{0/0/black/black}
		\end{tikzpicture}
	\end{minipage}
}
\def\textelbow{
	\begin{minipage}{13pt}
		\begin{tikzpicture}[scale=1]
			\tpipedream{0.4}{(0,0)}{0/0/black/black}
		\end{tikzpicture}
	\end{minipage}
}

\newcommand{\content}[3]{ 
  \coordinate (C) at #2;
	
  \foreach \X/\Y/\object in {#3} {
		\node at ($ (C) + ( #1 * \X, -#1 * \Y ) + ( #1 / 2, -#1 / 2 )$) {\object};
	}
}

\newcommand{\tpipedream}[3]{ 
  \coordinate (P) at #2;
	
  \foreach \x/\y/\a/\b in {#3} {
		\coordinate (P1) at ($ (P) + ( #1 * \x , -#1 * \y ) + ( 0      , #1 / 2 ) $);
	  \coordinate (P2) at ($ (P) + ( #1 * \x , -#1 * \y ) + ( #1     , #1 / 2 ) $);
	  \coordinate (P3) at ($ (P) + ( #1 * \x , -#1 * \y ) + ( #1 / 2 , #1     ) $);
	  \coordinate (P4) at ($ (P) + ( #1 * \x , -#1 * \y ) + ( #1 / 2 , 0 ) $);
	  \coordinate (P5) at ($ (P) + ( #1 * \x , -#1 * \y ) + ( #1 / 2 , #1 / 2 ) $);
	  \draw[rounded corners=4, color=\a, thick] (P1) -- (P5) -- (P3);
	  \draw[rounded corners=4, color=\b, thick] (P4) -- (P5) -- (P2);
	}

}

\newcommand{\cpipedream}[3]{ 
  \coordinate (P) at #2;
	
  \foreach \x/\y/\a/\b in {#3} {
		\coordinate (P1) at ($ (P) + ( #1 * \x , -#1 * \y ) + ( 0      , #1 / 2 ) $);
	  \coordinate (P2) at ($ (P) + ( #1 * \x , -#1 * \y ) + ( #1     , #1 / 2 ) $);
	  \coordinate (P3) at ($ (P) + ( #1 * \x , -#1 * \y ) + ( #1 / 2 , #1     ) $);
	  \coordinate (P4) at ($ (P) + ( #1 * \x , -#1 * \y ) + ( #1 / 2 , 0 ) $);
	  \coordinate (P5) at ($ (P) + ( #1 * \x , -#1 * \y ) + ( #1 / 2 , #1 / 2 ) $);
	  \draw[rounded corners=0.2, color=\a, thick] (P1) -- (P5) -- (P3);
	  \draw[rounded corners=0.2, color=\b, thick] (P4) -- (P5) -- (P2);
	}

}

\definecolor{grey}{gray}{0.6}

\title[Maximal $0$-$1$-fillings of moon polyominoes and
rc-graphs]{Maximal $0$-$1$-fillings of moon polyominoes with
  restricted chain lengths and rc-graphs}

\author{Martin Rubey}
\address{Institut f\"ur Algebra, Zahlentheorie und Diskrete Mathematik,Leibniz
  Universit\"at Hannover, Welfengarten 1, D-30167 Hannover, Germany}
\email{martin.rubey@math.uni-hannover.de}
\urladdr{http://www.iazd.uni-hannover.de/~rubey/}


\newtheorem{thm}{Theorem}[section]
\newtheorem{lem}[thm]{Lemma}

\newtheorem{prop}[thm]{Proposition}

\newtheorem{cnj}[thm]{Conjecture}

\theoremstyle{definition}
\newtheorem{dfn}[thm]{Definition}
 
\theoremstyle{remark}
\newtheorem{rmk}{Remark}

\newtheorem*{eg*}{Example}

\newcommand{\Set}[1]{\ensuremath{\mathcal{#1}}}            
\newcommand{\Dfn}[1]{\emph{#1}}                            
\newcommand{\Mat}[1]{\ensuremath{\mathbf{#1}}}             
\newcommand{\naturals}{\mathbb N}
   
\def\x{\mbox{$\bullet$}}
\def\g{\mbox{$\color{grey}\bullet$}}
\def\moon{{\tiny$\young(:\hfil\hfil,\hfil\hfil\hfil\hfil,\hfil\hfil\hfil\hfil,:\hfil\hfil)$}}

\def\adots{{.\hspace{1pt}\raisebox{2pt}{.}\hspace{1pt}\raisebox{4pt}{.}}}

\begin{document}
\maketitle

\begin{abstract}
  We show that maximal $0$-$1$-fillings of moon polynomials with
  restricted chain lengths can be identified with certain rc-graphs,
  also known as pipe dreams.  In particular, this exhibits a
  connection between maximal $0$-$1$-fillings of Ferrers shapes and
  Schubert polynomials.  Moreover, it entails a bijective proof
  showing that the number of maximal fillings of a stack polyomino
  $S$ with no north-east chains longer than $k$ depends only on $k$
  and the multiset of column heights of $S$.

  Our main contribution is a slightly stronger theorem, which in turn
  leads us to conjecture that the poset of rc-graphs with covering
  relation given by generalised chute moves is in fact a lattice.
\end{abstract}

\section{Introduction}
\label{sec:introduction}

\subsection{Triangulations, multitriangulations and $0$-$1$-fillings}
\sloppypar
The systematic study of $0$-$1$-fillings of polyominoes with
restricted chain lengths likely originates in an article by Jakob
Jonsson~\cite{Jonsson2005}.  At first, he was interested in a
generalisation of triangulations, where the objects under
consideration are maximal sets of diagonals of the $n$-gon, such that
at most $k$ diagonals are allowed to cross mutually.  Thus, in the
case $k=1$ one recovers ordinary triangulations.  He realised these
objects as fillings of the staircase shaped polyomino with
row-lengths $n-1, n-2,\dots,1$ with zeros and ones.  The condition
that at most $k$ diagonals cross mutually then translates into the
condition that the longest north-east chain in the filling has length
$k$, see Definition~\ref{dfn:fillings-and-chains}.  Instead of
studying fillings of the staircase shape only, he went on to consider
more general shapes which he called \Dfn{stack} and \Dfn{moon
  polyominoes}, see Definition~\ref{dfn:moon} and
Figure~\ref{fig:moon}.

For stack polyominoes he was able to prove that the number of maximal
fillings depends only on $k$ and the multiset of heights of the
columns, not on the particular shape of the polyomino.  He
conjectured that this statement holds more generally for moon
polyominoes, which was eventually proved by the
author~\cite{Rubey2006} using a technique introduced by Christian
Krattenthaler~\cite{Krattenthaler2006} based on Sergey Fomin's growth
diagrams for the Robinson-Schensted-Knuth correspondence.  However,
the proof given there is not fully bijective: what one would hope for
is a correspondence between fillings of any two moon polyominoes that
differ only by a permutation of the columns.  This article is a step
towards this goal.

\subsection{RC-graphs and the subword complex}

RC-graphs (for \lq reduced word compatible sequence graphs\rq,
see~\cite{MR1281474}, also known as \lq pipe dreams\rq\
see~\cite{MR2180402}) were introduced by Sergey Fomin and Anatol
Kirillov~\cite{MR1394950} to prove various properties of Schubert
polynomials.  Namely, for a given permutation $w$, the Schubert
polynomial $\mathfrak S_w$ can be regarded as the generating function
of rc-graphs, see Remark~\ref{rmk:Schubert}.

A different point of view is to consider them as facets of a certain
simplicial complex.  Let $w_0$ be the long permutation $n\cdots21$,
and consider its reduced factorisation
$$Q=s_{n-1}\cdots s_2 s_1\; s_{n-1}\cdots s_3 s_2\; \cdots\cdots\; s_{n-1}s_{n-2}\; s_{n-1}.$$
Then the subword complex associated to $Q$ and $w$ introduced by
Allen Knutson and Ezra Miller~\cite{MR2180402,MR2047852} has as
facets those subwords of $Q$ that are reduced factorisations of $w$.
Subword complexes enjoy beautiful topological properties, which are
transferred by the main theorem of this article to the simplicial
complex of $0$-$1$-fillings, as observed by Christian
Stump~\cite{Stump2010}, see also the article by Luis Serrano and
Christian Stump~\cite{SerranoStump2010}.

The intimate connection between maximal fillings and rc-graphs
demonstrated by the main theorem of this article,
Theorem~\ref{thm:filling-dream}, \emph{should} not have come as a
surprise.  Indeed, Sergey Fomin and Anatol Kirillov \cite{MR1471891}
established a connection between reduced words and reverse plane
partitions already thirteen years ago, which is not much less than
the case of Ferrers shapes in Theorem~\ref{thm:ne-se}.  They even
pointed towards the possibility of a bijective proof using the
Edelman-Greene correspondence.

More recently, the connection between Schubert polynomials and
triangulations was noticed by Alexander Woo~\cite{Woo2004}.  Vincent
Pilaud and Michel Pocchiola~\cite{PilaudPocchiola2009} discovered
rc-graphs (under the name \lq beam arrangements\rq) more generally
for multitriangulations, however, they were unaware of the theory of
Schubert polynomials.  In particular, Theorem 3.18 of Vincent
Pilaud's thesis~\cite{Pilaud2010} (see also Theorem~21
of~\cite{PilaudPocchiola2009}) is a variant of our
Theorem~\ref{thm:filling-dream} for multitriangulations.

Finally, Christian Stump and the author of the present article became
aware of an article by Vincent Pilaud and Francisco
Santos~\cite{MR2471876} that describes the structure of
multitriangulations in terms of so-called $k$-stars (introduced by
Harold Coxeter).  We then decided to translate this concept to the
language of fillings, and discovered pipe dreams yet again.

\section{Definitions}
\label{sec:definitions}

\subsection{Polyominoes}
\label{sec:polyominoes}
\begin{figure}[h]
  \begin{equation*}
  \begin{array}{ccc}
  \young(:::\hfil,%
         ::\hfil\hfil\hfil,%
         ::\hfil\hfil\hfil\hfil,%
         \hfil\hfil\hfil\hfil\hfil\hfil\hfil,%
         \hfil\hfil\hfil\hfil\hfil\hfil\hfil,%
         :\hfil\hfil\hfil\hfil\hfil\hfil,%
         :::\hfil\hfil)  
  &
  \young(\hfil\hfil\hfil\hfil\hfil\hfil\hfil,%
         \hfil\hfil\hfil\hfil\hfil\hfil\hfil,%
         ::\hfil\hfil\hfil\hfil,%
         ::\hfil\hfil\hfil,%
         :::\hfil)
  &
  \young(\hfil\hfil\hfil\hfil\hfil\hfil\hfil,%
         \hfil\hfil\hfil\hfil\hfil\hfil\hfil,%
         \hfil\hfil\hfil\hfil,%
         \hfil\hfil\hfil,%
         \hfil)
  \end{array}
  \end{equation*}
  \caption{a moon-polyomino, a stack-polyomino and a Ferrers diagram}
  \label{fig:moon}
\end{figure}

\begin{dfn}\label{dfn:polyominoes}
  A \Dfn{polyomino} is a finite subset of the quarter plane
  $\naturals^2$, where we regard an element of $\naturals^2$ as a
  cell.  A \Dfn{column} of a polyomino is the set of cells along a
  vertical line, a \Dfn{row} is the set of cells along a horizontal
  line.  We are using \lq English\rq\ (or matrix) conventions for the
  indexing of the rows and columns of polyominoes: the top row and
  the left-most column have index $1$.

  The polyomino is \Dfn{convex}, if for any two cells in a column
  (rsp. row), the elements of $\naturals^2$ in between are also cells
  of the polyomino.  It is \Dfn{intersection-free}, if any two
  columns are \Dfn{comparable}, {\it i.e.}, the set of row
  coordinates of cells in one column is contained in the set of row
  coordinates of cells in the other.  Equivalently, it is
  intersection-free, if any two rows are comparable.

  For example, the polyomino
  \begin{equation*}
    \young(::\hfil,%
    ::\hfil\hfil\hfil,%
    \hfil\hfil\hfil\hfil\hfil,%
    \hfil\hfil\hfil\hfil,%
    ::\hfil)  
  \end{equation*}
  is convex, but not intersection-free, since the first and the last
  columns are incomparable.
\end{dfn}

\begin{dfn}\label{dfn:moon}
  A \Dfn{moon polyomino} (or L-convex polyomino) is a convex,
  intersection-free polyomino.  Equivalently we can require that any
  two cells of the polyomino can be connected by a path consisting of
  neighbouring cells in the polyomino, that changes direction at most
  once.  A \Dfn{stack polyomino} is a moon-polyomino where all
  columns start at the same level.  A \Dfn{Ferrers diagram} is a
  stack-polyomino with weakly decreasing row widths
  $\lambda_1,\lambda_2,\dots,\lambda_n$, reading rows from top to
  bottom.

  Because a moon-polyomino is intersection free, the set of rows of
  maximal length in a moon polyomino must be consecutive.  We call
  the set of rows including these and the rows above the \Dfn{top
    half} of the polyomino.  Similarly, the set of columns of maximal
  length, and all columns to the right of these, is the \Dfn{right
    half} of the polyomino.  The intersection of the top and the
  right half is the \Dfn{top right quarter} of $M$.
\end{dfn}

\subsection{Fillings and Chains}
\label{sec:fillings-chains}
\begin{dfn}\label{dfn:fillings-and-chains}
  A \Dfn{$0$-$1$-filling} of a polyomino is an assignment of numbers
  $0$ and $1$ to the cells of the polyomino.  Cells containing $0$
  are also called \Dfn{empty}.

  A \Dfn{north-east chain} is a sequence of non-zero entries in a
  filling such that the smallest rectangle containing all its
  elements is completely contained in the moon polyomino and such
  that for any two of its elements one is strictly to the right and
  strictly above the other.
\end{dfn}
As it turns out, it is more convenient to draw dots instead of ones
and leave cells filled with zeros empty.  Two examples of (rather
special) fillings of a moon polyomino are depicted in
Figure~\ref{fig:top-bot}.  In both examples the length of the longest
north-east chain is $2$.

\begin{dfn}
  $\Set F_{01}^{ne}(M, k)$ is the set of $0$-$1$-fillings of the moon
  polyomino $M$ whose longest north-east chain has length $k$ and
  that are \Dfn{maximal}, {\it i.e.}, assigning an empty cell a $1$
  would create a north-east chain of length $k+1$.  For a vector
  $\Mat r$ of integers, $\Set F_{01}^{ne}(M, k, \Mat r)$ is the
  subset of $\Set F_{01}^{ne}(M, k)$ consisting of those fillings
  that have exactly $\Mat r_i$ zero entries in row $i$.

  For any filling in $\Set F_{01}^{ne}(M, k)$, and an empty cell
  $\epsilon$, there must be a chain $C$ such that replacing the $0$
  with $1$ in $\epsilon$, and adding $\epsilon$ to $C$, would make
  $C$ into a $(k+1)$-chain.  In this situation, we say that $C$ is a
  \Dfn{maximal chain for} $\epsilon$.
\end{dfn}
For example, when $M$ is the moon polyomino \moon, the set $\Set
F_{01}^{ne}(M, 1)$ consists of ten fillings, as can be inferred from
Figure~\ref{fig:poset}.

\begin{rmk}
  Note that extending the first $k$ rows and columns of a Ferrers
  diagram does not affect the set $\Set F_{01}^{ne}$, which is why we
  choose to fix the number of zero entries instead of entries equal
  to $1$, although the latter might seem more natural at first
  glance.
\end{rmk}
\begin{figure}
  \centering
\begin{tikzpicture}[scale=0.6]
\node (v1) at (2.5cm, 5.0cm)      [draw=none] {$1$};
\node (v7) at (0.4952cm,4.0097cm) [draw=none] {$7$};
\node (v6) at (0.0cm,1.7845cm)    [draw=none] {$6$};
\node (v5) at (1.3874cm,0.0cm)    [draw=none] {$5$};
\node (v4) at (3.6126cm,0.0cm)    [draw=none] {$4$};
\node (v3) at (5.0cm,1.7845cm)    [draw=none] {$3$};
\node (v2) at (4.5048cm,4.0097cm) [draw=none] {$2$};
\draw [thick,grey] (v1) to (v2);
\draw [thick,grey] (v1) to (v3);
\draw [thick]      (v1) to (v5);
\draw [thick,grey] (v1) to (v6);
\draw [thick,grey] (v1) to (v7);
\draw [thick,grey] (v2) to (v3);
\draw [thick,grey] (v2) to (v4);
\draw [thick]      (v2) to (v5);
\draw [thick,grey] (v2) to (v7);
\draw [thick,grey] (v3) to (v4);
\draw [thick,grey] (v3) to (v5);
\draw [thick]      (v3) to (v6);
\draw [thick]      (v3) to (v7);
\draw [thick,grey] (v4) to (v5);
\draw [thick,grey] (v4) to (v6);
\draw [thick,grey] (v5) to (v6);
\draw [thick,grey] (v5) to (v7);
\draw [thick,grey] (v6) to (v7);
\content{0.78}{(6.7,5.5)}{%
  0/0/$1$,1/0/$2$,2/0/$3$,3/0/$4$,4/0/$5$,5/0/$6$,
  -1/1/$7$,-1/2/$6$,-1/3/$5$,-1/4/$4$,-1/5/$3$,-1/6/$2$}
\node at (9,2.4)
{\young(\g\g\x\hfil\g\g,\g\hfil\x\g\g,\x\x\g\g,\hfil\g\g,\g\g,\g)};
\end{tikzpicture}

\begin{tikzpicture}[scale=0.95]
\node (v1) at (1cm, 6cm)          [draw=none, grey] {$\bullet$};
\node (v2) at (1.5cm, 6cm)        [draw=none, grey] {$\bullet$};
\node (v3) at (2cm, 6cm)        [draw=none] {$\bullet$};
\node (v4) at (2.5cm, 6cm)        [draw=none] {$\bullet$};
\node (v5) at (3cm, 6cm)        [draw=none] {$\bullet$};
\node (v6) at (3cm, 5.5cm)        [draw=none] {$\bullet$};
\node (v7) at (3.5cm, 5.5cm)        [draw=none] {$\bullet$};
\node (v8) at (3.5cm, 5cm)        [draw=none] {$\bullet$};
\node (v9) at (3.5cm, 4.5cm)        [draw=none] {$\bullet$};
\node (v10) at (3.5cm, 4cm)        [draw=none, grey] {$\bullet$};
\node (v11) at (3.5cm, 3.5cm)        [draw=none, grey] {$\bullet$};
\node (w1) at (1.5cm, 5.5cm)          [draw=none] {$\bullet$};
\node (w2) at (2cm, 5.5cm)        [draw=none] {$\bullet$};
\node (w3) at (2cm, 5cm)        [draw=none] {$\bullet$};
\node (w4) at (2.5cm, 5cm)        [draw=none] {$\bullet$};
\node (w5) at (3cm, 5cm)        [draw=none] {$\bullet$};
\node (w6) at (3cm, 4.5cm)        [draw=none] {$\bullet$};
\node (w7) at (3cm, 4cm)        [draw=none] {$\bullet$};
\draw [thick, grey] (1cm,6cm) -- (2cm,6cm);
\draw [thick] (2cm,6cm) -- (3cm,6cm) -- (3cm,5.5cm) -- (3.5cm,5.5cm) -- (3.5cm,4.5cm);
\draw [thick] (1.5cm,5.5cm) -- (2cm,5.5cm) -- (2cm,5cm) -- (3cm,5cm) -- (3cm,4cm);
\draw [thick, grey]  (3.5cm,4.5cm) -- (3.5cm,3.5cm);
\node at (6.65, 4.8)
{\young(\g\g\x\x\x\hfil,:\x\x\hfil\x\x,::\x\x\x\x,:::\hfil\x\x,::::\x\g,:::::\g)};
\end{tikzpicture}
\caption{a $2$-triangulation with corresponding filling of the
  staircase $\lambda_0$ and a fan of two Dyck paths with
  corresponding filling of the reverse staircase $\lambda_0^{rev}$.}
   \label{fig:triangulation-Dyck}
\end{figure}

\begin{rmk}
  For the staircase shape $\lambda_0$ with $n-1$ rows the set $\Set
  F_{01}^{ne}(\lambda_0, k)$ has a particularly beautiful
  interpretation, namely as the set of $k$-triangulations of the
  $n$-gon.  More precisely, label the vertices of the $n$-gon
  clockwise from $1$ to $n$, and identify a cell of the shape in row
  $i$ and column $j$ with the pair $(n-i+1, j)$ of vertices.  Thus,
  the entries in the filling equal to $1$ define a set of diagonals
  of the $n$-gon.  It is not hard to check that a north-east chain of
  length $k$ in the filling corresponds to a set of $k$ mutually
  crossing diagonals in the $n$-gon.

  Maximal fillings of the reverse staircase shape $\lambda_0^{rev}$
  for a given $k$ are in bijection with fans of $k$ Dyck paths.  An
  illustration of both correspondences is given in
  Figure~\ref{fig:triangulation-Dyck}.  These correspondences were
  Jakob Jonsson's~\cite{Jonsson2005} starting point to prove (in a
  quite non-bijective fashion) that there are as many
  $k$-triangulations of the $n$-gon as fans of $k$ non-intersecting
  Dyck paths with $n-2k$ up steps each.  Luis Serrano and Christian
  Stump~\cite{SerranoStump2010} provided the first completely
  bijective proof of this fact, which we generalise in
  Section~\ref{sec:Edelman-Greene}.  Remarkably, Alex
  Woo~\cite{Woo2004} used the same methods already much earlier for
  the case of triangulations and Dyck paths, {\it i.e.}, $k=1$.
\end{rmk}

\subsection{Pipe dreams}
In this section we collect some results around pipe dreams and
rc-graphs.  All of these statements can be found in~\cite{MR1281474}
together with precise references.
\begin{figure}
  \centering
  \begin{tikzpicture}
    \tpipedream{0.475}{(1.95, 0.6875)}{%
      0/0/black/black,1/0/black/black,3/0/black/black,4/0/black/black,5/0/black/black,6/0/black/white,%
      0/1/black/black,1/1/black/black,4/1/black/black,5/1/black/white,%
      1/2/black/black,3/2/black/black,4/2/black/white,%
      0/3/black/black,3/3/black/white,%
      0/4/black/black,1/4/black/black,2/4/black/white,%
      0/5/black/black,1/5/black/white,%
      0/6/black/white%
    }%
    \cpipedream{0.475}{(1.95, 0.6875)}{%
      2/0/black/black,2/1/black/black,3/1/black/black,0/2/black/black,%
      2/2/black/black,1/3/black/black,2/3/black/black}%
    \content{0.475}{(1.95, 1.6375)}{%
      0/0/$1$,1/0/$2$,2/0/$3$,3/0/$4$,4/0/$5$,5/0/$6$,6/0/$7$,
      -1/1/$1$,-1/2/$2$,-1/3/$6$,-1/4/$4$,-1/5/$7$,-1/6/$5$,-1/7/$3$}%
    \content{0.475}{(5.95, 1.1625)}{%
      0/0/\x,1/0/\x,3/0/\x,4/0/\x,5/0/\x,6/0/\x,%
      0/1/\x,1/1/\x,4/1/\x,5/1/\x,%
      1/2/\x,3/2/\x,4/2/\x,%
      0/3/\x,3/3/\x,%
      0/4/\x,1/4/\x,2/4/\x,%
      0/5/\x,1/5/\x,%
      0/6/\x,%
      2/0/+,2/1/+,3/1/+,0/2/+,%
      2/2/+,1/3/+,2/3/+}%
    \content{0.475}{(5.95, 1.6375)}{%
      0/0/$1$,1/0/$2$,2/0/$3$,3/0/$4$,4/0/$5$,5/0/$6$,6/0/$7$,
      -1/1/$1$,-1/2/$2$,-1/3/$6$,-1/4/$4$,-1/5/$7$,-1/6/$5$,-1/7/$3$}%
    \content{0.475}{(9.95, 1.1625)}{%
      0/0/\x,1/0/\x,3/0/\x,4/0/\x,5/0/\x,6/0/\x,%
      0/1/\x,1/1/\x,4/1/\x,5/1/\x,%
      1/2/\x,3/2/\x,4/2/\x,%
      0/3/\x,3/3/\x,%
      0/4/\x,1/4/\x,2/4/\x,%
      0/5/\x,1/5/\x,%
      0/6/\x,%
      2/0/3,2/1/4,3/1/5,0/2/3,%
      2/2/5,1/3/5,2/3/6}%
    \content{0.475}{(9.95, 1.6375)}{%
      0/0/$1$,1/0/$2$,2/0/$3$,3/0/$4$,4/0/$5$,5/0/$6$,6/0/$7$,
      -1/1/$1$,-1/2/$2$,-1/3/$6$,-1/4/$4$,-1/5/$7$,-1/6/$5$,-1/7/$3$}
\end{tikzpicture}
\caption{the reduced pipe dream associated to the reduced
  factorisation $s_3 s_5 s_4 s_5 s_3 s_6 s_5$ of $1,2,6,4,7,5,3$.}
  \label{fig:dreams}
\end{figure}
\begin{dfn}\label{dfn:pipe} 
  A \Dfn{pipe dream} for a permutation $w$ is a filling of a the
  quarter plane $\naturals^2$, regarding each element of
  $\naturals^2$ as a cell, with \Dfn{elbow joints} $\textelbow$ and a
  finite number of \Dfn{crosses} $\textcross$, such that a pipe
  entering from above in column $i$ exits to the left from row
  $w^{-1}(i)$.  A pipe dream is \Dfn{reduced} if each pair of pipes
  crosses at most once, it is then also called \Dfn{rc-graph}.
  $\Set{RC}(w)$ is the set of reduced pipe dreams for $w$, and, for a
  vector $\Mat r$ of integers, $\Set{RC}(w, \Mat r)$ is the subset of
  $\Set{RC}(w)$ having precisely $\Mat r_i$ crosses in row $i$.
\end{dfn}
  Usually it will be more convenient to draw dots instead of elbow
  joints and sometimes to omit crosses.  We will do so without
  further notice.

\begin{rmk}
  We can associate a reduced factorisation of $w$ to any pipe dream
  in $\Set{RC}(w)$ as follows: replace each cross appearing in row
  $i$ and column $j$ of the pipe dream with the elementary
  transposition $(i+j-1, i+j)$.  Then the reduced factorisation of
  $w$ is given by the sequence of transpositions obtained by reading
  each row of the pipe dream from right to left, and the rows from
  top to bottom.  An example can be found in Figure~\ref{fig:dreams},
  where we write $s_i$ for the elementary transposition $(i,i+1)$.
\end{rmk}

\begin{rmk}\label{rmk:Schubert}
  Using reduced pipe dreams, it is possible to define the Schubert
  polynomial $\mathfrak S_w$ for the permutation $w$ in a very
  concrete way.  For a reduced pipe dream $D\in\Set{RC}(w)$, define
  $x^D=\prod_{(i,j)\in D} x_i$, where the product runs over all
  crosses in the pipe dream.  Then the Schubert polynomial is just
  the generating function for pipe dreams:
  \begin{equation*}
    \mathfrak S_w = \sum_{D\in\Set{RC}(w)} x^D.
  \end{equation*}
  This definition of Schubert polynomials and their evaluation by
  Sergey Fomin and Anatol Kirillov~\cite{MR1471891} was used by
  Christian Stump~\cite{Stump2010} to give a simple proof of the
  product formula for the number of $k$-triangulations of the $n$-gon
  \begin{equation*}
    \prod_{1\leq i,j<n-2k} \frac{i+j+2k}{i+j}.
  \end{equation*}
\end{rmk}

We now define an operation on pipe dreams which was introduced in a
slightly less general form by Nantel Bergeron and Sara
Billey~\cite{MR1281474}.  It will be the main tool in the proof of
Theorem~\ref{thm:filling-dream}.
\begin{dfn}
  Let $D\in\Set{RC}(w)$ be a pipe dream.  Then a \Dfn{chute move} is
  a modification of $D$ of the following form:
  \begin{equation*}
  \begin{array}{@{}c@{}}\\[-5ex]
    \begin{array}{@{}r|c|c|c|c|c|l@{}}
      \multicolumn{5}{c}{}&\multicolumn{1}{c}{
        \phantom{+}}&
      \multicolumn{1}{c}{\begin{array}{@{}c@{}}\\\adots\end{array}}
      \\\cline{2-6}
      &\x&+&\cdots&+&+\\\cline{2-3}\cline{5-6}
      &+ &+&\cdots&+&+\\\cline{2-3}\cline{5-6}
      &\multicolumn{5}{c|}{\vdots\hfill\vdots\hfill\vdots\hfill}&\\\cline{2-3}\cline{5-6}
      &+ &+&\cdots&+&+\\\cline{2-3}\cline{5-6}
      &\x&+&\cdots&+&\x\\\cline{2-6}
      \multicolumn{1}{c}{\begin{array}{@{}c@{}}\adots\\ \\ \end{array}}&
      \multicolumn{1}{c}{\phantom{+}}
    \end{array}
    \quad\stackrel{\text{chute}}\rightsquigarrow\quad
    \begin{array}{@{}r|c|c|c|c|c|l@{}}
      \multicolumn{5}{c}{}&\multicolumn{1}{c}{
        \phantom{+}}&
      \multicolumn{1}{c}{\begin{array}{@{}c@{}}\\\adots\end{array}}
      \\\cline{2-6}
      &\x&+&\cdots&+&\x\\\cline{2-3}\cline{5-6}
      &+ &+&\cdots&+&+\\\cline{2-3}\cline{5-6}
      &\multicolumn{5}{c|}{\vdots\hfill\vdots\hfill\vdots\hfill}&\\\cline{2-3}\cline{5-6}
      &+ &+&\cdots&+&+\\\cline{2-3}\cline{5-6}
      &+ &+&\cdots&+&\x\\\cline{2-6}
      \multicolumn{1}{c}{\begin{array}{@{}c@{}}\adots\\ \\ \end{array}}&
      \multicolumn{1}{c}{\phantom{+}}
    \end{array}
    \\[-3ex]
  \end{array}
  \end{equation*}
  More formally, a \Dfn{chutable rectangle} is a rectangular region
  $r$ inside a pipe dream $D$ with at least two columns and two rows
  such that all but the following three locations of $r$ are crosses:
  the north-west, south-west, and south-east corners.  Applying a
  \Dfn{chute move} to $D$ is accomplished by placing a \textcross\ in
  the south-west corner of a chutable rectangle $r$ and removing the
  \textcross\ from the north-east corner of $r$.  We call the inverse
  operation \Dfn{inverse chute move}.
\end{dfn}

The following lemma was given by Nantel Bergeron and Sara
Billey~\cite[Lemma~3.5]{MR1281474} for two rowed chute moves, the
proof is valid for our generalised chute moves without modification:
\begin{lem}\label{lem:chute-closure}%
  The set $\Set{RC}(w)$ of reduced pipe dreams for~$w$ is closed
  under chute moves.
\end{lem}
\begin{proof}
  The pictorial description of chute moves in terms of pipes
  immediately shows that the permutation associated to the pipe dream
  remains unchanged.  For example, here is the picture associated
  with a three rowed chute move:
  \begin{equation*}
    \begin{tikzpicture}[scale=0.88]
      \tpipedream{0.5}{(0,0)}{%
        0/0/black/black,%
        0/2/black/black,7/2/black/black}%
      \cpipedream{0.5}{(0,0)}{%
        1/0/grey/grey,2/0/grey/grey,3/0/grey/grey,4/0/grey/grey,5/0/grey/grey,6/0/grey/grey,7/0/grey/grey,%
        0/1/grey/grey,1/1/grey/grey,2/1/grey/grey,3/1/grey/grey,4/1/grey/grey,5/1/grey/grey,6/1/grey/grey,7/1/grey/grey,%
        1/2/grey/grey,2/2/grey/grey,3/2/grey/grey,4/2/grey/grey,5/2/grey/grey,6/2/grey/grey}%
    \end{tikzpicture}
    \quad\raisebox{0.5cm}{$\stackrel{\text{chute}}\rightsquigarrow$}\quad
    \begin{tikzpicture}[scale=0.88]
      \tpipedream{0.5}{(0,0)}{%
        0/0/black/black,7/0/black/black,%
        7/2/black/black}%
      \cpipedream{0.5}{(0,0)}{%
        1/0/grey/grey,2/0/grey/grey,3/0/grey/grey,4/0/grey/grey,5/0/grey/grey,6/0/grey/grey,%
        0/1/grey/grey,1/1/grey/grey,2/1/grey/grey,3/1/grey/grey,4/1/grey/grey,5/1/grey/grey,6/1/grey/grey,7/1/grey/grey,%
        0/2/grey/grey,1/2/grey/grey,2/2/grey/grey,3/2/grey/grey,4/2/grey/grey,5/2/grey/grey,6/2/grey/grey}%
    \end{tikzpicture}
  \end{equation*}
\end{proof}
\begin{rmk}
  It follows that chute moves define a partial order on $\Set{RC}(w)$,
  where $D$ is covered by $E$ if there is a chute move transforming
  $E$ into $D$.  Nantel Bergeron and Sara Billey restricted their
  attention to two rowed chute moves.  For this case, their main
  theorem states that the poset defined by chute moves has a unique
  maximal element, namely
  $$
  D_{top}(w)=\left\{(c,j): %
    c\leq \#\{i: i < w^{-1}_j, w_i>j\}\right\}.
  $$
  It is easy to see that considering general chute moves, the poset
  has also a unique minimal element, namely
  $$
  D_{bot}(w)=\left\{(i,c): %
    c\leq \#\{j: j > i, w_j < w_i\}\right\}.
  $$
  In the next section we will show a statement similar in spirit to
  the main theorem of Nantel Bergeron and Sara Billey for the more
  general chute moves defined above.
\end{rmk}

\begin{figure}
\begin{center}
\small
\setlength{\arraycolsep}{0.6ex}
\def\lr#1{\multicolumn{1}{|c|}{\raisebox{-.3ex}{$#1$}}}
\def\lrg#1{\multicolumn{1}{|c|}{\raisebox{-.3ex}{\cellcolor[gray]{0.7}$#1$}}}
\def\csix{\hhline{------}}
\def\cfiv{\hhline{-----}}
\def\cfou{\hhline{----}}
\def\cthr{\hhline{---}}
\def\ctwo{\hhline{--}}
\def\cone{\hhline{-}}
\def\thick{}
\scalebox{0.4}{
}
\end{center}
\caption{the poset of reduced pipe dreams for the permutation $1, 2,
  6, 4, 5, 3$.  The interval of $0$-$1$-fillings with $k=1$ of the
  moon polyomino \protect\moon\ is emphasised.}
  \label{fig:poset}
\end{figure}

After generating and analysing some of these posets using
\texttt{Sage}~\cite{Sage-Combinat}, see Figure~\ref{fig:poset} for an
example, we became convinced that they should have much more
structure:
\begin{cnj}\label{cnj:lattice}
  The poset of reduced pipe dreams defined by (general) chute moves
  is in fact a lattice.
\end{cnj}

There is another natural way to transform one reduced pipe dream into
another, originating in the concept of flipping a diagonal of a
triangulation.  Namely, consider an elbow joint in the pipe dream.
Since any pair of pipes crosses at most once, there is at most one
location where the pipes originating from the given elbow joint
cross.  If there is such a crossing, replace the elbow joint by a
cross and the cross by an elbow joint.  Clearly, the result is again
a reduced pipe dream, associated to the same permutation.

It is believed (see Vincent Pilaud and Michel
Pocchiola~\cite{PilaudPocchiola2009}, Question~51) that the
simplicial complex of multitriangulations can be realised as a
polytope, in this case the graph of flips would be the graph of the
polytope.  Note that the graph of chute moves is a subgraph of the
graph of flips.  Is Conjecture~\ref{cnj:lattice} related to the
question of polytopality?

\section{Maximal Fillings of Moon Polyominoes and Pipe Dreams}
\label{sec:maximal-fillings-rc}

Consider a maximal filling in $\Set F_{01}^{ne}(M, k)$.  Recall that
we regard a moon polyomino $M$ as a finite subset of $\naturals^2$.
Also, recall that a pipe dream is nothing but a filling of
$\naturals^2$ with elbow joints and a finite number of crosses.
Thus, replacing zeros in the filling of the moon polyomino with
crosses, and all cells in the filling containing ones as well as all
cells not in $M$ with elbow joints, we clearly obtain a pipe dream
for some permutation $w$.  An example of this transformation is given
in Figure~\ref{fig:pipe-dream-filling}.  We will see in this section
that the pipe dreams obtained in this way are in fact reduced.

\begin{figure}
  \centering
    \begin{tikzpicture}
    \node at (0,0)
    {\young(:\x\hfil,\x\x\hfil\hfil,\hfil\x\hfil\x,:\hfil\hfil\x,:\x\x)};
    \node at (2.8975,-0.005)
    {\young(:\hfil\hfil,\hfil\hfil\hfil\hfil,\hfil\hfil\hfil\hfil,:\hfil\hfil\hfil,:\hfil\hfil)};
    \content{0.475}{(1.95, 1.1625)}{%
      0/0/\x,1/0/\x,3/0/\x,4/0/\x,5/0/\x,6/0/\x,%
      0/1/\x,1/1/\x,4/1/\x,5/1/\x,%
      1/2/\x,3/2/\x,4/2/\x,%
      0/3/\x,3/3/\x,%
      0/4/\x,1/4/\x,2/4/\x,%
      0/5/\x,1/5/\x,%
      0/6/\x,%
      2/0/+,2/1/+,3/1/+,0/2/+,%
      2/2/+,1/3/+,2/3/+}%
    \node at (6.8975,-0.005)
    {\young(:\hfil\hfil,\hfil\hfil\hfil\hfil,\hfil\hfil\hfil\hfil,:\hfil\hfil\hfil,:\hfil\hfil)};
    \tpipedream{0.475}{(5.95, 0.6875)}{%
      0/0/black/black,1/0/black/black,3/0/black/black,4/0/black/black,5/0/black/black,6/0/black/white,%
      0/1/black/black,1/1/black/black,4/1/black/black,5/1/black/white,%
      1/2/black/black,3/2/black/black,4/2/black/white,%
      0/3/black/black,3/3/black/white,%
      0/4/black/black,1/4/black/black,2/4/black/white,%
      0/5/black/black,1/5/black/white,%
      0/6/black/white%
    }%
    \cpipedream{0.475}{(5.95, 0.6875)}{%
      2/0/black/black,2/1/black/black,3/1/black/black,0/2/black/black,%
      2/2/black/black,1/3/black/black,2/3/black/black}%
    \content{0.475}{(5.95, 1.6375)}{%
      0/0/$1$,1/0/$2$,2/0/$3$,3/0/$4$,4/0/$5$,5/0/$6$,6/0/$7$,
      -1/1/$1$,-1/2/$2$,-1/3/$6$,-1/4/$4$,-1/5/$7$,-1/6/$5$,-1/7/$3$}%
  \end{tikzpicture}
  \caption{a maximal filling and the associated pipe dream.}
  \label{fig:pipe-dream-filling}
\end{figure}

One may notice that the permutation associated with the pipe dream so
constructed depends somewhat on the embedding of the polyomino into
the quarter plane.  Although one can check that this dependence is
not substantial for what is to follow, we will assume for simplicity
that the top row and the left-most column of the polyomino have index
$1$ and indices increase from top to bottom and from left to right.

Even without the knowledge that the pipe dream is reduced we can
speak of chute moves applied to fillings in $\Set F_{01}^{ne}(M, k)$.
However, a priori it is not clear under which conditions the result
of such a move is again a filling in $\Set F_{01}^{ne}(M, k)$.  In
particular, we have to deal with the fact that under this
identification all cells outside $M$ are also filled with \emph{elbow
  joints}, corresponding to \emph{ones}.  Of course, to determine the
set of north-east chains we have to consider the original filling and
the boundary of $M$, and \emph{disregard} elbow joints outside.

Similar to the approach of Nantel Bergeron and Sara Billey we will
consider two special fillings $D_{top}(M, k)$ and $D_{bot}(M, k)$.
These will turn out to be the maximal and the minimal element in the
poset having elements $\Set F_{01}^{ne}(M, k)$, where one filling is
smaller than another if it can be obtained by applying chute moves to
the latter.  Figure~\ref{fig:top-bot} displays an example of the
following construction:
\begin{dfn}
  Let $M$ be a moon polyomino and $k\geq0$.  Then $D_{top}(M,
  k)\in\Set F_{01}^{ne}(M, k)$ is obtained by putting ones into all
  cells that can be covered by any rectangle of size at most $k\times
  k$, which is completely contained in the moon polyomino, and that
  touches the boundary of $M$ with its lower-left corner.

  Similarly, $D_{bot}(M, k)\in\Set F_{01}^{ne}(M, k)$ is obtained by
  putting ones into all cells that can be covered by any rectangle of
  size at most $k\times k$, which is completely contained in the moon
  polyomino, and that touches the boundary of $M$ with its
  upper-right corner.
\end{dfn}
\begin{figure}[h]
  \begin{equation*}
    \young(:::\x\x\hfil\hfil,%
    ::\x\x\hfil\hfil\hfil,%
    ::\x\x\hfil\hfil\hfil,%
    \x\x\x\hfil\hfil\hfil\hfil\hfil,%
    \x\x\x\hfil\hfil\hfil\hfil\hfil,%
    :\x\x\x\x\hfil\hfil\x,%
    :\x\x\x\x\x\x\x,%
    :::\x\x\x\x)%
    \quad
    \young(:::\x\x\x\x,%
    ::\x\x\x\x\x,%
    ::\x\hfil\hfil\x\x,%
    \x\x\hfil\hfil\hfil\x\x\x,%
    \x\x\hfil\hfil\hfil\x\x\x,%
    :\hfil\hfil\hfil\hfil\hfil\x\x,%
    :\hfil\hfil\hfil\hfil\hfil\x\x,%
    :::\hfil\hfil\x\x)%
  \end{equation*}
  \caption{The special fillings $D_{top}(M,k)$ and $D_{bot}(M,k)$ for
    $k=2$ of a moon polyomino.}
  \label{fig:top-bot}
\end{figure}

We can now state the main theorem of this article:
\begin{thm}\label{thm:filling-dream}\sloppypar
  Let $M$ be a moon polyomino and $k\geq 0$.  The set $\Set
  F_{01}^{ne}(M, k, \Mat r)$ can be identified with the set of
  reduced pipe dreams $\Set{RC}\big(w(M, k), \Mat r\big)$ having all
  crosses inside of $M$ for some permutation $w(M, k)$ depending only
  on $M$ and $k$: replace zeros with crosses and all cells containing
  ones as well as all cells not in $M$ with elbow joints.

  More precisely, the set $\Set F_{01}^{ne}(M, k)$ is an interval in
  the poset of reduced pipe dreams $\Set{RC}\big(w(M, k)\big)$ with
  maximal element $D_{top}(M, k)$ and minimal element $D_{bot}(M,
  k)$.
\end{thm}

As already remarked in the introduction various versions of this
theorem were independently proved by various authors by various
methods.  The most general version is due to Luis Serrano and
Christian Stump~\cite[Theorem~2.6]{SerranoStump2010}, whose proof
employs properties of subword complexes and who thus obtain
additionally many interesting properties of the simplicial complex of
$0$-$1$-fillings.

The advantage of our approach using chute moves is the demonstration
of the property that $\Set F_{01}^{ne}(M, k)$ is in fact an interval
in the bigger poset of reduced pipe dreams.  In particular, if
Conjecture~\ref{cnj:lattice} turns out to be true then $\Set
F_{01}^{ne}(M, k)$ is also a lattice.  An illustration is given in
Figure~\ref{fig:poset}.

Let us first state a very basic property of chute moves as applied to
fillings:
\begin{lem}\label{lem:chute-moon-closure}
  Let $M$ be a moon polyomino.  Chute moves and their inverses
  applied to a filling in $\Set F_{01}^{ne}(M, k)$ produce another
  filling in $\Set F_{01}^{ne}(M, k)$ whenever all zero entries
  remain in $M$.
\end{lem}
\begin{proof}
  We only have to check that chain lengths are preserved, which is
  not hard.
\end{proof}

Most of what remains of this section is devoted to prove that there
is precisely one filling in $\Set F_{01}^{ne}(M, k)$ that does not
admit a chute move such that the result is again in $\Set
F_{01}^{ne}(M, k)$, namely $D_{bot}(M, k)$, and precisely one filling
that does not admit an inverse chute move with the same property,
namely $D_{top}(M, k)$.

Although the strategy itself is actually very simple the details turn
out to be quite delicate.  Thus we split the proof into a few
auxiliary lemmas.  Let us fix $k$, a moon polyomino $M$, and a
maximal filling $D\in\Set F_{01}^{ne}(M, k)$ different from
$D_{bot}(M, k)$.  We will then explicitly locate a chutable
rectangle.  Throughout the proof maximality of the filling will play
a crucial role.  The first lemma is used to show that certain cells
of the polyomino must be empty because otherwise the filling would
contain a chain of length $k+1$:
\begin{lem}[Chain induction]\label{lem:chain-induction}
  Consider a maximal filling of a moon polyomino.  Let $\epsilon$ be
  an empty cell such that all cells below $\epsilon$ in the same
  column are empty too, except possibly those that are below the
  lowest cell of the column left of $\epsilon$.  Assume that for
  \emph{each} of these cells $\delta$ there is a maximal chain for
  $\delta$ strictly north-east of $\delta$.  Then there is a maximal
  chain for $\epsilon$ strictly north-east of $\epsilon$.

  Similarly, let $\epsilon$ be an empty cell such that all cells left
  of $\epsilon$ in the same row are empty too, except possibly those
  that are left of the left-most cell of the row below $\epsilon$.
  Assume that for \emph{each} of these cells $\delta$ there is a
  maximal chain for $\delta$ strictly north-east of $\delta$.  Then
  there is a maximal chain for $\epsilon$ strictly north-east of
  $\epsilon$.
\end{lem}
\begin{rmk}
  Note that for the conclusion of Lemma~\ref{lem:chain-induction} to
  hold we really have to assume that \emph{all} cells below
  $\epsilon$ are empty: in the maximal filling for $k=1$
  \begin{equation*}
    \young(\x\epsilon\x,%
    \hfil\delta\x,%
    \x\x)
  \end{equation*}
  there is a maximal chain for $\delta$ north-east of
  $\delta$, but no maximal chain for $\epsilon$ north-east of
  $\epsilon$.  The following example demonstrates that it is equally
  necessary that the filling is maximal:
  \begin{equation*}
    \young(\x,\epsilon\x,\hfil\x)
  \end{equation*}
\end{rmk}
\begin{proof}
  Assume on the contrary that there is no maximal chain for
  $\epsilon$ north-east of $\epsilon$.  Consider a maximal chain
  $C_\epsilon$ for $\epsilon$ that has as many elements north-east of
  $\epsilon$ as possible.  Let $\delta$ be the cell in the same
  column as $\epsilon$, below $\epsilon$, in the same row as the top
  entry of $C_\epsilon$ which is south-east of $\epsilon$.  By
  assumption, there is a maximal chain $C_\delta$ for $\delta$
  north-east of $\delta$.  We have to consider two cases:

  If the widest rectangle containing $C_\epsilon$ is not as wide as
  the smallest rectangle containing $C_\delta$, then the entry of
  $C_\epsilon$ to the left of $\delta$ would extend $C_\delta$ to a
  $(k+1)$-chain, which is not allowed:
  \begin{center}
    \setlength{\unitlength}{0.5cm}
    \begin{picture}(12,11)
      \put(2,0){\framebox(8,10){}}         
      \put(10,10){$C_\epsilon$}
      \put(0,3){\framebox(12,5){}}         
      \put(12,8){$C_\delta$}
      \put(5.5,0){\dashbox{0.3}(1,10){}}     
      \put(5.5,3){\framebox(1,1){$\delta$}}
      \put(5.5,6){\framebox(1,1){$\epsilon$}}
      \put(4,3){\makebox(1,1){$\x$}}
      \put(3.5,1.8){\makebox(1,1){$\x$}}
      \put(2.5,1){\makebox(1,1){$\adots$}}
      \put(7,7){\makebox(1,1){$\x$}}
      \put(7.6,8){\makebox(1,1){$\x$}}
      \put(8.5,8.5){\makebox(1,1){$\adots$}}
      \put(7,4){\makebox(1,1){$\x$}}
      \put(8.8,5){\makebox(1,1){$\adots$}}
      \put(11,5.8){\makebox(1,1){$\x$}}
    \end{picture}
  \end{center}

  If the smallest rectangle containing $C_\epsilon$ is at least as
  wide as the widest rectangle containing $C_\delta$, then we obtain
  a maximal chain for $\epsilon$ north-east of $\epsilon$ by
  induction.  Let $c_\epsilon^1, c_\epsilon^2,\dots$ be the sequence
  of elements of $C_\epsilon$ north-east of $\epsilon$, and
  $c_\delta^1, c_\delta^2,\dots$ the sequence of elements of
  $C_\delta$ north-east of $\delta$.  We will show that
  $c_\epsilon^i$ must be strictly north and weakly west of
  $c_\delta^i$, for all $i$.  Thus, the elements $c_\epsilon^1,
  c_\epsilon^2,\dots$ together with the elements of $C_\delta$
  outside the smallest rectangle containing $C_\epsilon$ form a
  maximal chain for $\epsilon$ north-east of $\epsilon$.

  $c_\epsilon^1$ is strictly north of $c_\delta^1$, since otherwise
  $C_\delta$ would be a maximal chain for $\epsilon$.  $c_\epsilon^1$
  cannot be strictly east of $c_\delta^1$, since in this case
  $c_\delta^1$ together with $C_\epsilon$ would be a $(k+1)$-chain.

  Suppose now that $c_\epsilon^{i-1}$ is strictly north and weakly
  west of $c_\delta^{i-1}$.  $c_\delta^i$ cannot be strictly
  north-east of $c_\epsilon^{i-1}$, since this would yield a
  $k$-chain north-east of $\epsilon$.  $c_\delta^i$ must be strictly
  east of $c_\epsilon^{i-1}$, since $c_\delta^i$ is strictly east of
  $c_\delta^{i-1}$, which in turn is weakly east of
  $c_\epsilon^{i-1}$ by the induction hypothesis.  Thus,
  $c_\epsilon^{i-1}$ is weakly north and strictly west of
  $c_\delta^i$.

  $c_\epsilon^i$ cannot be strictly north-east of $c_\delta^i$, since
  then the elements of $C_\epsilon$ south-west of $\epsilon$ together
  with the elements $c_\delta^1,\dots,c_\delta^i$ and
  $c_\epsilon^i,c_\epsilon^{i+1}, \dots$ would form a $(k+1)$-chain.
  Finally, $c_\epsilon^i$ must be strictly north of $c_\delta^i$,
  since $c_\epsilon^i$ is strictly north of $c_\epsilon^{i-1}$, which
  in turn is weakly north of $c_\delta^i$.
\end{proof}

\begin{lem}\label{lem:chutable-rectangle}
  Consider a maximal filling of a moon polyomino.  Suppose that there
  is a rectangle with at least two columns and at least two rows
  completely contained in the polyomino, with all cells empty except
  the north-west, south-east and possibly the south-west corners.
  Then the south-west corner is indeed non-empty, {\it i.e.}, the
  rectangle is chutable.
\end{lem}
Note that we must insist that the south-west corner of the rectangle
is part of the polyomino.  Here is a maximal filling with $k=1$,
where the three cells in the south-west do not form a chutable
rectangle, since the south-west corner is missing:
\begin{equation*}
  \young(:\x\hfil\x,%
  \x\hfil\hfil\hfil,%
  \x\hfil\hfil\x,%
  :\x\x)%
\end{equation*}
However, we can weaken this assumption in a different way:
\begin{lem}\label{lem:chutable-rectangle-2}
  Consider a maximal filling of a moon polyomino.  Suppose that there
  is a rectangle with at least two columns and at least two rows such
  that all cells of its top row and its right column are contained in
  the polyomino.  Assume furthermore that all cells of the rectangle
  that are in the polyomino are empty except the north-west,
  south-east and possibly the south-west corners.  Finally, suppose
  that there is no maximal chain for the cell in the north-east
  corner strictly north east of it.  Then the cell in the south-west
  corner is indeed in the polyomino and non-empty, {\it i.e.}, the
  rectangle is chutable.
\end{lem}
\begin{proof}[Proof of Lemma~\ref{lem:chutable-rectangle}]
  Suppose on the contrary that the cell in the south-west corner is
  empty, too.  Then, the situation is as in the following picture:
  \begin{center}
    \setlength{\unitlength}{0.5cm}
    \begin{picture}(6,4)%
      \put(0,0){\framebox(6,4){}} %
      \put(0,0){\framebox(1,1){$\delta$}} %
      \put(5,0){\framebox(1,1){$\x$}} %
      \put(5,3){\framebox(1,1){$\epsilon$}}%
      \put(0,3){\framebox(1,1){$\x$}} %
    \end{picture}
  \end{center}

  Since the filling is maximal but the cells $\delta$ and $\epsilon$
  are empty, there must be maximal chains for these cells.  The
  corresponding rectangles must not cover any of the two cells
  containing ones, since that would imply the existence of a
  $(k+1)$-chain.  Thus, any maximal chain for $\delta$ must be
  strictly south-west of $\delta$, and any maximal chain for
  $\epsilon$ must be strictly north-east of $\epsilon$.  Since the
  polyomino is intersection free, the top row of the rectangle
  containing the maximal chain for $\epsilon$ is either contained in
  the bottom row of the rectangle containing the maximal chain for
  $\delta$, or vice versa.  In both cases, we have a contradiction.
\end{proof}

The next lemma parallels the main Lemma~3.6 in the article by Nantel
Bergeron and Sara Billey~\cite{MR1281474}:
\begin{lem}\label{lem:two-column-chute}
  Consider a maximal filling of a moon polyomino.  Suppose that there
  is a cell $\gamma$ containing a $1$ with an empty cell $\epsilon$
  in the neighbouring cell to its right, such that there are at least
  as many cells above $\gamma$ as above $\epsilon$.  Then the filling
  contains a chutable rectangle.

  Similarly, suppose that there is a cell $\gamma$ containing a $1$
  with an empty cell $\epsilon$ in the neighbouring cell below it,
  such that there are at least as many cells right of $\gamma$ as
  right of $\epsilon$.  Then the filling contains a chutable
  rectangle.
\end{lem}
\begin{proof}
  Suppose that all of the cells in the column containing $\epsilon$,
  which are below $\epsilon$ and weakly above the bottom cell of the
  column containing $\gamma$, are empty.  Let $\delta$ be the lowest
  cell in this region.  There must then be a maximal chain for
  $\delta$ that is north-east of $\delta$.  By
  Lemma~\ref{lem:chain-induction}, we conclude that there is also a
  maximal chain for $\epsilon$ north-east of $\epsilon$.  However,
  then the $1$ in the cell left of $\epsilon$ together with this
  chain yields a $(k+1)$-chain, since the rectangle containing the
  maximal chain for $\epsilon$ extends by hypothesis to the column
  left of $\epsilon$.

  We can thus apply Lemma~\ref{lem:chutable-rectangle} to the
  following rectangle: the south-east corner being the top non-empty
  cell below $\epsilon$, and the north-west corner being the lowest
  cell containing a $1$ in the column of $\gamma$, strictly above the
  chosen south-east corner.
\end{proof}

Finally, the main statement follows from a careful analysis of
fillings different from $D_{bot}(M, k)$, repeatedly applying the
previous lemmas to exclude obstructions to the existence of a
chutable rectangle:
\begin{thm}\label{thm:maximal-fillings-chutable}
  Any maximal filling other than $D_{bot}(M,k)$ admits a chute move
  such that the result is again a filling of $M$.  Any maximal
  filling other than $D_{top}(M,k)$ admits an inverse chute move such
  that the result is again a filling of $M$.
\end{thm}
\begin{proof}
  Suppose that all cells in the top-right quarter of $M$ that contain
  a $1$ in $D_{bot}(M,k)$ also contain a $1$ in the filling $F$ at
  hand.  It follows, that all cells that are empty in $D_{bot}(M,k)$
  are empty in $F$, too, because there is a maximal chain for each of
  them.  Thus, in this case $F=D_{bot}(M,k)$.

  Otherwise, consider the set of left-most cells in the top-right
  quarter, that contain a $1$ in $D_{bot}(M,k)$ but are empty in $F$,
  and among those the top cell, $\epsilon$.  If its left or lower
  neighbour contains a $1$, we can apply
  Lemma~\ref{lem:two-column-chute} and are done.  Otherwise, we have
  to find a rectangle as in the hypothesis of
  Lemma~\ref{lem:chutable-rectangle}.  The difficulty in this
  undertaking is to prove that the lower left corner is indeed part
  of the polyomino.  To ease the understanding of the argument, we
  will frequently refer to the following sketch:

  \begin{center}
    \setlength{\unitlength}{0.5cm}
    \begin{picture}(21,12)
      \put(0,9){\line(1,0){2}}
      \put(2,9){\line(0,1){1}}
      \put(2,10){\line(1,0){1.5}}
      \multiput(3.5,10)(0.5,0){35}{\line(1,0){0.1}}
      \put(21,10.1){$R$}
      \put(3.8,10.3){\adots}
      \put(5,11){\line(1,0){1}}
      \put(6,11){\line(0,1){1}}
      \multiput(6,0)(0,0.5){24}{\line(0,1){0.25}}
      \multiput(6,12)(0.5,0){8}{\line(1,0){0.25}}
      \put(1,7){\framebox(1,1){$\x^\alpha$}}
      \put(2,7){$\overbrace{\makebox(4,1){}}^\ell$}
      \put(1,7){\dashbox{0.3}(15,1){}}
      \put(10,7){\framebox(1,1){$\epsilon$}}
      \put(10,8){\framebox(4,4){$k\times k$}}
      \put(10,8){\makebox(1,1){$\x$}}
      \put(10,11){\makebox(1,1){$\x$}}
      \put(13,8){\makebox(1,1){$\x$}}
      \put(13,11){\makebox(1,1){$\x$}}
      \put(10,0){\framebox(1,1){$\x^\beta$}}
      \put(10,0){\dashbox{0.3}(1,7){}}
      \put(3,5){\framebox(1,1){$\x^{\alpha'}$}}
      \put(12,2){\framebox(1,1){$\x^{\beta'}$}}
      \put(3,2){\framebox(1,1){$\omega$}}
      \put(12,5){\framebox(1,1){$\delta$}}
      \put(16,7){\framebox(4.7,3.7){$X$}}
      \put(16,7){\makebox(1,1){$\x$}}
      \put(16,9.7){\makebox(1,1){$\x$}}
      \put(19.7,7){\makebox(1,1){$\x$}}
      \put(19.7,9.7){\makebox(1,1){$\x$}}
      \put(14,8){\makebox(1,1){$\x$}}
      \put(15,8){\makebox(1,1){$\x$}}
      \put(14,9.7){\makebox(1,1){$\x$}}
      \put(15,9.7){\makebox(1,1){$\x$}}
    \end{picture}
  \end{center}

  By construction, there is a $k\times k$ square filled with ones
  just above $\epsilon$, and there is no $k$-chain strictly
  north-east of $\epsilon$.  This implies in particular that the top
  cell in the left-most column of the polyomino must be lower than
  the top row of the $k\times k$ square, because otherwise there
  could not be any maximal chain for $\epsilon$.

  By Lemma~\ref{lem:chain-induction} there must therefore be a
  non-empty cell left of $\epsilon$, which we label $\alpha$, and a
  non-empty cell below $\epsilon$, which we label $\beta$.  Note that
  there may be entries to the right of $\epsilon$, in the same row,
  which are non-empty.  However, we can assume that to the right of
  the first such entry all other cells in this row are non-empty,
  too, because otherwise we could apply
  Lemma~\ref{lem:two-column-chute}.

  We can now construct a chutable rectangle: let $\beta'$ be the top
  cell containing a $1$ below an empty cell weakly to the right of
  $\epsilon$, and if there are several, the left-most.  Also, let
  $\alpha'$ be the lowest cell among the right-most containing a $1$,
  which are weakly below $\alpha$, but strictly above $\beta'$.  Let
  $\delta$ be the cell in the same row as $\alpha'$ and the same
  column as $\beta'$.  Let $\omega$ be the cell in the same column as
  $\alpha'$ and the same row as $\beta'$ -- a priori we do not know
  however that $\omega$ is a cell in the polyomino.  We then apply
  Lemma~\ref{lem:chutable-rectangle-2} to the rectangle defined by
  $\alpha'$ and the first non-empty cell to the right of $\omega$, in
  the same row.

  To achieve our goal we show that there cannot be a maximal chain
  for $\delta$ north-east of $\delta$.  Suppose on the contrary that
  there is such a chain.  At least its top-right element must be in a
  row (denoted $R$ in the sketch) strictly above the top cell of the
  column containing $\alpha'$: otherwise, $\alpha'$ together with
  this chain would form a $(k+1)$-chain.  In the sketch, the
  non-empty cells that are implied are indicated by the rectangle
  denoted $X$, which must be of size $k\times k$ at least.

  If the size of the region to the right of $\alpha$ indicated by
  $\ell$ in the sketch is $0$ we let $\sigma$ be the bottom-left cell
  of a maximal chain for $\epsilon$.  Otherwise, define $\sigma$ to
  be the bottom-left cell of a maximal chain for the right neighbour
  of $\alpha$.  In both cases, $\sigma$ must be in a column weakly
  west of $\alpha$ -- in the latter case because there are fewer than
  $k$ cells above the right neighbour of $\alpha$.  

  By intersection free-ness applied to the column containing $\sigma$
  and the columns containing cells of $X$, the latter columns must
  all extend at least down to the row containing $\sigma$.  But in
  this case the maximal chain with bottom left cell $\sigma$ can be
  extended to a $(k+1)$-chain using the cells in $X$.

  We have shown that a maximal chain for $\delta$ must have some
  elements south-east of $\delta$.  We can now apply
  Lemma~\ref{lem:chutable-rectangle-2}.
\end{proof}

\begin{proof}[Proof of Theorem~\ref{thm:filling-dream}]
  All pipe dreams in $\Set{RC}(w)$ contained in $M$ are maximal
  $0$-$1$ fillings of $M$, since they can be generated by applying
  sequences of chute moves to $D_{top}(M,k)$.

  Since we can apply chute moves to any maximal $0$-$1$-filling of
  $M$ except $D_{bot}(M,k)$, all such fillings arise in this fashion.
  (We have to remark here that in case the pipe dream associated to
  some filling would not be reduced, applying chute moves eventually
  exhibits that the filling was not maximal.)  Together with
  Lemma~\ref{lem:chute-moon-closure}, this implies that all fillings
  $F_{01}^{ne}(M, k)$ have the same associated permutation.

  Note that this procedure implies, as a by-product, that all maximal
  $0$-$1$-fillings of $M$ have the same number of entries equal to
  zero, {\it i.e.}, the simplicial complex of $0$-$1$-fillings is
  pure.
\end{proof}

\section{Applying the Edelman-Greene correspondence}
\label{sec:Edelman-Greene}

Using the identification described in the previous section, we can
apply a correspondence due to Paul Edelman and Curtis
Greene~\cite{MR871081}, that associates pairs of tableaux to reduced
factorisations of permutations.  This in turn will yield the desired
bijective proof of Jakob Jonsson's result at least for stack
polyominoes.

The main result of this section was obtained for Ferrers shapes
earlier by Luis Serrano and Christian Stump~\cite{SerranoStump2010}
using the same proof strategy.  For stack polyominoes the description
of the $P$-tableau is different, thus we believe it is useful to
repeat the arguments here.

The following theorem is a collection of results from Paul Edelman
and Curtis Greene~\cite{MR871081}, Richard Stanley~\cite{MR782057}
and Alain Lascoux and Marcel-Paul Sch\"utzenberger~\cite{MR686357},
and describes properties of the \Dfn{Edelman-Greene} correspondence:
\begin{thm}\label{thm:edelman-greene}
  Let $w$ be a permutation and $s_i$ be the elementary transposition
  $(i, i+1)$.  Consider pairs of words $(u,v)$ of the same length
  $\ell$, such that $s_{v_1},s_{v_2},\dots,s_{v_\ell}$ is a reduced
  factorisation of $w$ and $u_i\leq u_{i+1}$, with $u_i=u_{i+1}$ only
  if $v_i>v_{i+1}$.

  There is a bijection between such pairs of words and pairs $(P, Q)$
  of Young tableaux of the same shape, such that $P$ is column and
  row strict and whose reading word is a reduced factorisation of
  $w$, and such that the transpose of $Q$ is semistandard.  Moreover,
  if $w$ is vexillary, {\it i.e.}, $2143$-avoiding, the tableau $P$
  is the same for all reduced factorisations of $w$.
\end{thm}
This correspondence can be defined via row insertion.  We insert a
letter $x$ into row $r$ of a tableau $P$ whose last letter is
different from $x$ as follows: if $x$ is (strictly) greater than all
letters in row $r$, we just append $x$ to row $r$.  If row $r$
contains both the letters $x$ and $x+1$ we insert $x+1$ into row
$r+1$.  Otherwise, let $y$ be the smallest letter in row $r$ that is
strictly greater than $x$, replace $y$ in row $r$ by $x$ and insert
$y$ into row $r+1$.

We can now construct the pair of tableaux $(P, Q)=(P_\ell, Q_\ell)$
from a pair of words $(u, v)$ as in the statement of
Theorem~\ref{thm:edelman-greene}: let $P_0$ and $Q_0$ be empty
tableaux.  Insert the letter $v_i$ into the first row of $P_{i-1}$ to
obtain $P_i$, and place the letter $u_i$ into the cell of $Q_i$
determined by the condition that $P_i$ and $Q_i$ have the same shape.

It turns out that the permutations associated to moon polyominoes are
indeed vexillary:
\begin{prop}\label{prop:vexillary}
  For any moon-polyomino $M$ and any $k$ the permutation $w(M, k)$ is
  vexillary.
\end{prop}
\begin{rmk}
  There are vexillary permutations which do not correspond to moon
  polyominoes.  For example, the only two reduced pipe dreams for the
  permutation $4,2,5,1,3$ are
  \begin{equation*}
    \begin{tikzpicture}
      \content{0.475}{(0, 0)}{%
        0/0/+,1/0/+,2/0/+,3/0/\x,4/0/\x,%
        0/1/+,1/1/\x,2/1/+,3/1/\x,%
        0/2/+,1/2/\x,2/2/\x,%
        0/3/\x,1/3/\x,%
        0/4/\x}%
      \content{0.475}{(0, 0.475)}{%
        0/0/$1$,1/0/$2$,2/0/$3$,3/0/$4$,4/0/$5$,
        -1/1/$4$,-1/2/$2$,-1/3/$5$,-1/4/$1$,-1/5/$3$}%
    \end{tikzpicture}
    \quad\raisebox{40pt}{\text{and}}\quad
    \begin{tikzpicture}
      \content{0.475}{(0, 0)}{%
        0/0/+,1/0/+,2/0/+,3/0/\x,4/0/\x,%
        0/1/+,1/1/\x,2/1/\x,3/1/\x,%
        0/2/+,1/2/+,2/2/\x,%
        0/3/\x,1/3/\x,%
        0/4/\x}%
      \content{0.475}{(0, 0.475)}{%
        0/0/$1$,1/0/$2$,2/0/$3$,3/0/$4$,4/0/$5$,
        -1/1/$4$,-1/2/$2$,-1/3/$5$,-1/4/$1$,-1/5/$3$}%
    \end{tikzpicture}
  \end{equation*}

\end{rmk}
\begin{proof}
  It is sufficient to prove the claim for $k=0$, since the empty cells
  in the filling $D_{top}(M, k)$ for any $k$ again form a moon
  polyomino.  Thus, suppose that the permutation associated to $M$ is
  not vexillary.  Then we have indices $i<j<k<\ell$ such that
  $w(j)<w(i)<w(\ell)<w(k)$.  It follows that the pipes entering in
  columns $i$ and $j$ from above cross, and so do the two pipes
  entering in columns $k$ and $\ell$, and thus correspond to cells of
  the moon polyomino.  Since any two cells in the moon polyomino can
  be connected by a path of neighbouring cells changing direction at
  most once, there is a third cell where either the pipes entering
  from $i$ and $\ell$ or from $j$ and $k$ cross, which is impossible.
\end{proof}

\begin{thm}[for Ferrers shapes, Luis Serrano and Christian Stump \cite{SerranoStump2010}]\label{thm:ne-se}
  Consider the set $\Set F_{01}^{ne}(S, k, \Mat r)$, where $S$ is a
  stack polyomino.  Let $\mu_i$ be the number of cells the
  $i$\textsuperscript{th} row of $S$ is indented to the right, and
  suppose that $\mu_1=\dots =\mu_k=\mu_{k+1}=0$.

  Let $u$ be the word $1^{\Mat r_1}, 2^{\Mat r_2},\dots$ and let $v$
  be the reduced factorisation of $w$ associated to a given pipe
  dream.  Then the Edelman-Greene correspondence applied to the pair
  of words $(u, v)$ induces a bijection between $\Set F_{01}^{ne}(S, k,
  \Mat r)$ and the set of pairs $(P, Q)$ of Young tableaux satisfying
  the following conditions:
  \begin{itemize}
  \item the common shape of $P$ and $Q$ is the multiset of column
    heights of the empty cells in $D_{top}(S, k)$,
  \item the first row of $P$ equals $(k+1, k+2+\mu_{k+2},
    k+3+\mu_{k+3},\dots)$, and the entries in columns are
    consecutive,
  \item $Q$ has type $\{1^{\Mat r_1}, 2^{\Mat r_2},\dots\}$, and
    entries in column $i$ are at most $i+k$.
  \end{itemize}

  Thus, the common shape of $P$ and $Q$ encodes the row lengths of
  $S$, the entries of the first row of $P$ encode the left border
  of $S$, and the entries of $Q$ encode the filling.
\end{thm}
\begin{rmk}
  In particular, this theorem implies an explicit bijection between
  the sets $\Set F_{01}^{ne}(S_1, k, \Mat r)$ and $\Set
  F_{01}^{ne}(S_2, k, \Mat r)$, given that the multisets of column
  heights of $S_1$ and $S_2$ coincide.

  Curiously, the most natural generalisation of the above theorem to
  moon polyominoes is not true.  Namely, one may be tempted to
  replace the condition on $Q$ by requiring that the entries of $Q$
  are between $Q_{top}$ and $Q_{bot}$ component-wise.  However, this
  fails already for $k=1$ and the shape
  \begin{equation*}
    \Yvcentermath0
    \young(:\hfil\hfil,%
    :\hfil\hfil,%
    \hfil\hfil\hfil,%
    \hfil\hfil\hfil)\,\,,
  \end{equation*}
  with $P=\young(345,5)$, $Q_{top}=\young(123,3)$ and
  $Q_{bot}=\young(234,4)$.  In this case, the tableau
  $Q=\young(124,3)$ has preimage
  \begin{equation*}
    \Yvcentermath0
    \young(:\x\hfil,%
    :\x\x\hfil,%
    \x\hfil\x,%
    \x\hfil\x).
  \end{equation*}
\end{rmk}
\begin{rmk}
  One might hope to prove Conjecture~\ref{cnj:lattice} by applying
  the Edelman-Greene correspondence, and checking that the poset is a
  lattice on the tableaux.  However, at least for the natural
  component-wise order on tableaux, the correspondence is not order
  preserving, not even for the case of Ferrers shapes.
\end{rmk}
\begin{proof}
  In view of Proposition~\ref{prop:vexillary}, to obtain the tableau
  $P$ it is enough to insert the reduced word given by the filling
  $D_{top}(S, k)$ using the Edelman-Greene correspondence, which is
  not hard for stack polyominoes.

  It remains to prove that the entries in column $i$ of $Q$ are at
  most $i+k$ precisely if $(u,v)$ comes from a filling in $\Set
  F_{01}^{ne}(S, k)$.  To this end, observe that the shape of the
  first $i$ columns of $P$ equals the shape of the tableau obtained
  after inserting the pair of words $\left((u_1,
    u_2,\dots,u_\ell),(v_1, v_2,\dots,v_\ell)\right)$, where $\ell$
  is such that $u_\ell\leq k+i$ and $u_{\ell+1}>k+i$.

  Namely, this is the case if and only if the first $i+k+\mu_{i+k+1}$
  positions of the permutation corresponding to $(v_1,
  v_2,\dots,v_\ell)$ coincide with those of the permutation $w$
  corresponding to $v$ itself, as can be seen by considering
  $D_{top}(w)$, whose empty cells form again a stack polyomino.

  This in turn is equivalent to all letters $v_m$ being at least
  $k+i+1+\mu_{k+i+1}$ for $m>\ell$, {\it i.e.}, whenever the
  corresponding empty cell of the filling occurs in a row below the
  $(i+k)$\textsuperscript{th} of $S$, and thus, when it is inside
  $S$.
\end{proof}

\section*{Acknowledgements}

I am very grateful to my wife for encouraging me to write this note,
and for her constant support throughout.  I would also like to thank
Thomas Lam and Richard Stanley for extremely fast replies concerning
questions about Theorem~\ref{thm:edelman-greene}.

I would like to acknowledge that Christian Stump provided a
preliminary version of \cite{Stump2010}.  Luis Serrano and Christian
Stump informed me privately that they were able to prove that all
$k$-fillings of Ferrers shapes yield the same permutation $w$,
however, their ideas would not work for stack polyominoes.  I was
thus motivated to attempt the more general case.
\providecommand{\cocoa} {\mbox{\rm C\kern-.13em o\kern-.07em C\kern-.13em
  o\kern-.15em A}}
\providecommand{\bysame}{\leavevmode\hbox to3em{\hrulefill}\thinspace}
\providecommand{\MR}{\relax\ifhmode\unskip\space\fi MR }
\providecommand{\MRhref}[2]{%
  \href{http://www.ams.org/mathscinet-getitem?mr=#1}{#2}
}
\providecommand{\href}[2]{#2}

\end{document}